\documentclass[11pt]{amsart}

\usepackage{amsmath, amsthm, amssymb, amsfonts, enumerate}
\usepackage[colorlinks=true,linkcolor=blue,urlcolor=blue]{hyperref}
\usepackage[normalem]{ulem}
\usepackage{color}
\usepackage{geometry}
\usepackage{soul}
\usepackage{graphicx}
\newtheorem{theorem}{Theorem}[section]
\newtheorem{remark}[theorem]{Remark}
\newtheorem{definition}[theorem]{Definition}
\newtheorem{assumption}[theorem]{Assumption}
\newtheorem{lemma}[theorem]{Lemma}

\newtheorem{corollary}[theorem]{Corollary}

\def \cA{{\mathcal A}}

\def \cC{{\mathcal C}}

\def \cF{{\mathcal F}}

\def \cL{{\mathcal L}}
\def \cO{{\mathcal O}}

\def \cU{{\mathcal U}}

\def \cS{{\mathcal S}}
\def \E{\mathsf{E}}
\def \P{\mathsf{P}}
\def \R{\mathbb{R}}

\def \N{\mathbb{N}}

\def \eps{\varepsilon}

\def \ud{{\mathrm d}}
\def \e{{\mathrm e}}

\DeclareMathOperator*{\supp}{supp}
\newcommand\indd[1]{1_{\{#1\}}}

\DeclareMathOperator*{\graph}{Gr}

\newcommand{\rom}[1]{\uppercase\expandafter{\romannumeral #1\relax}}

\definecolor{brightmaroon}{rgb}{0.7, 0.23, 0.2}

\newcommand\jp[1]{{#1}}

\title[Continuity of optimal surfaces]{On the continuity of optimal stopping surfaces \\ for jump-diffusions}
\author[C.\ Cai, T.\ De Angelis, J.\ Palczewski]{Cheng Cai \and Tiziano De Angelis \and Jan Palczewski}
\address{C.\ Cai: School of Mathematics, University of Leeds, Woodhouse Lane, LS2 9JT Leeds, UK}
\email{\href{mailto: mmcca@leeds.ac.uk}{mmcca@leeds.ac.uk}}
\address{T.\ De Angelis: School of Management and Economics, Dept.\ ESOMAS, University of Torino, Corso Unione Sovietica, 218 Bis, 10134, Torino, Italy; Collegio Carlo Alberto, Piazza Arbarello 8, 10122, Torino, Italy.}
\email{\href{mailto:tiziano.deangelis@unito.it}{tiziano.deangelis@unito.it}}
\address{J.\ Palczewski: School of Mathematics, University of Leeds, Woodhouse Lane, LS2 9JT Leeds, UK}
\email{\href{mailto: j.palczewski@leeds.ac.uk}{j.palczewski@leeds.ac.uk}}
\date{\today}

\numberwithin{equation}{section}

\begin{document}

\begin{abstract}
We show that optimal stopping surfaces $(t,y)\mapsto x_*(t,y)$ arising from time-inhomogeneous optimal stopping problems 
on two-dimensional jump-diffusions $(X,Y)$ are continuous (jointly in time and space) under mild \jp{monotonicity and regularity} assumptions of local nature. 
\end{abstract}

\maketitle

\section{Introduction}
In recent years the literature on optimal stopping for multi-dimensional processes (often discontinuous) has experienced a steady growth. Motivations stem from ever more complex applications in mathematical finance (e.g., \cite{DT02,LT19,CDeAP21}), economics (e.g., \cite{DK20,DeAFF17}) and optimal detection/prediction (e.g., \cite{EPZ20,EW21}), among other fields. Particular interest has been devoted to the Markovian framework, where the theory of free boundary problems provides powerful tools for a detailed characterisation of optimal stopping rules. 
This elegant connection between probability and analysis is enabled by variational principles that reduce the stopping problem to the question of solving suitable partial differential equations (PDEs). It turns out that the optimal stopping boundary in the original stopping problem coincides with the free boundary of the domain for the PDE. So, both probabilistic and analytical methods have been widely employed for the study of such boundaries.

Establishing the continuity of the optimal stopping boundary has been historically one of the key challenges in the field. On an abstract mathematical level this is of interest because the regularity of the stopping boundary has direct bearing on the regularity of the value function of the optimal stopping problem and, therefore, on the possibility to establish existence and uniqueness of solutions for the associated free boundary problem. On a more practical level, the continuity of the optimal boundary is generally required for efficient numerical computation of the boundary itself. Indeed, an optimal stopping boundary is very often computed as the unique solution of a non-linear integral equation of Fredholm or Volterra type. In that context, continuity is needed to establish {\em uniqueness} of the solution to the integral equation, using methods originally developed in \cite{P05am} and subsequently refined/adapted in various other papers (see \cite{PS} for numerous examples and further references). The use of such integral equations is nowadays a standard machinery, which originated with ideas contained in \cite{K56, MK65, VM76} and later flourished in the context of American option pricing \cite{J91,K90,M92}. \jp{Recently, it was shown that integral equations can also be used in multi-dimensional diffusive optimal stopping problems \cite{CDeAP21,CCMS}, thus motivating our study of the optimal boundary's continuity in higher dimensions}.

Our contribution to the field is twofold and our results are obtained under very mild conditions that only pertain local properties of the problem data. First, we establish continuity of optimal stopping boundaries that depend on time {\em and} on another state-variable. Existing results cover (almost exclusively) the case of boundaries depending on a single variable (either `time' or `space' as in \cite{DeA15,LM} or \cite{P19}) -- clearly, the continuity of a surface is a much subtler issue than the continuity of a curve. Second, our results hold for a broad class of optimal stopping problems, including the case of jump-diffusion models. Previous work focussed mainly on continuous processes with only few exceptions in the context of American option pricing; these include \cite{LM} which considers a L\'evy process, and \cite{BX09,YJB06} which consider jump-diffusion  processes. Differently from our set-up, in those papers the optimal boundary is a function of time only.

From a methodological point of view our approach is conceptually easy \jp{and it relies upon the use of the adjoint operator of the underlying diffusion's infinitesimal generator}. Its roots are in ideas used in \cite{LM, DeA15} but only for time-dependent boundaries.
Our extension to optimal stopping surfaces requires non-trivial additional work and we are also able to relax smoothness assumptions on the coefficients of the underlying dynamics. Our work also \jp{generalises \cite[Thm.\ 10 and Cor.\ 20]{P19}} which consider boundaries that are functions of one variable only (either `time' or `space'). \jp{As in \cite{P19} we use smooth-fit to rule out discontinuities of the boundary but our 3-dimensional set-up induces important technical differences that we discuss in more detail in Remarks \ref{rem:goran1}--\ref{rem:goran3}.}

Regularity of free boundaries has been studied, of course, in the PDE literature. The focus there is generally on applications of free boundary problems to physical phenomena as, e.g., the celebrated Stefan's ice-melting problem (see \cite{Cannon} and \cite{Fri} for more details and references; see also \cite{Fri75} for a classical reference in the case of parabolic problems with one spatial dimension). Some of the ideas and methods from those areas of application also spilled over to the mathematical finance community in the context of American put/call option pricing. When the boundary is a function of time only and the underlying stochastic process is a Brownian motion (or a geometric Brownian motion), it can be shown that the optimal boundary is continuously differentiable (or better) using fine estimates on the Gaussian (or log-normal) transition density (see, e.g., \cite{CC07}). Similar ideas are also used in \cite{BX09} in the presence of jumps. That approach however does not extend to higher dimensions. 

In a multi-dimensional framework, with either elliptic or parabolic operators, continuous differentiability or higher regularity of the free boundary can be obtained via sophisticated PDE methods that often require a special structure of the second order operator (normally a Laplacian); see, e.g., \cite{LS09}. The use of those methods in the optimal stopping literature has been quite limited, partly due to the higher complexity of the underlying stochastic models. In the context of multi-dimensional optimal stopping problems, Lipschitz continuity of the boundary was shown in \cite{DeAS19} under some restrictive assumptions on the problem data and only for continuous processes.

As pointed out in the introduction of \cite{P19}, both in the PDE and probabilistic literature, continuity or higher regularity of the free boundary is mainly studied in specific examples using ad-hoc methods. In this paper, instead, we continue the development of a general theory for broader applications by addressing the problem in a general framework. This is in the spirit of earlier contributions \cite{DeA15,LM,P19} which we \jp{continue} by increasing the dimension of the state space while at the same time including jumps in the underlying dynamics. 

The paper is organised as follows. We set-up the problem in Section \ref{sec:setting}. For the clarity of exposition, we prove continuity of the optimal boundary first in the case of continuous processes (Section \ref{sec:con-1}) and then in fuller generality (Section \ref{sec:con-jump}). In all cases, we provide a detailed discussion around sufficient conditions that guarantee our assumptions hold in practical applications. We conclude in Section \ref{sec:last} by illustrating extensions to more general dynamics and gain functions featuring additive processes as, for example, local times.

\section{Setting}\label{sec:setting}
On a complete filtered probability space $(\Omega,\cF, (\cF_t),\P)$ we consider a 
two-dimensional time-inhomogeneous strong Markov process $(X,Y)$ with state space $\cO \subset \R^2$, where $\cO$ is open and possibly unbounded. The infinitesimal generator is
defined via its action on smooth functions $f:[0,T]\times \cO \to \R$ as follows:
\begin{align}
\label{eq:generator}
(\cL f)(t,x,y)=(\cL^\circ f)(t,x,y)+(\cA f)(t,x,y).
\end{align}
Here we have set
\begin{align}\label{eq:cL}
(\cL^\circ f)(t,x,y):=&\big(\beta_1\partial_{xx}f +\beta_2\partial_{yy}f + 2\bar \beta\partial_{xy}f+\alpha_1\partial_x f+\alpha_2\partial_y f\big)(t,x,y),
\end{align}
for some Borel-measurable functions \jp{$\alpha_i: [0, T] \times \cO \to \R$, $\beta_i: [0, T] \times \cO \to [0, \infty)$,} with 
\[
\bar \beta(t,x,y)=\rho\sqrt{\beta_1(t,x,y)\beta_2(t,x,y)}
\]
and $\rho \in [-1, 1]$. Letting $E$ be the space of Borel measurable functions $[0,T] \times \cO \to \R$, the operator $\cA:D_\cA\to E$ is a linear map defined on a linear subspace  $D_\cA\subset E$. We avoid making specific (global) regularity assumptions on the coefficients $\alpha_i, \beta_i$, but we will later require some local properties thereof. For the sake of mathematical generality we also do not specify the form of $\cA$ any further but it will be clear in a moment that our main application is for jump-diffusion processes. We use the notation $\P_{t,x,y}(\,\cdot\,):=\P(\,\cdot\,|X_t=x,Y_t=y)$ and $\E_{t,x,y}[\,\cdot\,]$, for the expectation under $\P_{t,x,y}$.

Letting $T\in(0,\infty]$ be the time horizon, for $(t,x,y)\in [0,T]\times\cO$ we are interested in optimal stopping problems of the form
\begin{equation}
\label{eq:osp-1}
\begin{aligned}
v(t,x,y)=\sup_{t\le \tau\le T}\E_{t,x,y}\left[\e^{-\int_t^\tau r(s,X_{s},Y_{s})\ud s}g(\tau,X_{\tau}, Y_{\tau})\right],
\end{aligned}
\end{equation}
where the supremum is taken over all $(\cF_t)$-stopping times. The discount rate $r:[0,T]\times\cO\to[0, \infty)$ and the gain function $g:[0,T]\times\cO\to\R$ are Borel-measurable functions and we will impose further assumptions later on as necessary. 

We assume that the problem is {\em well-posed} in the sense that the value function is finite for all $(t,x,y)\in[0,T]\times\cO$, the stopping time 
\[
\tau_*:=\inf\{s\in[t,T] : v(s,X_{s},Y_{s})=g(s,X_{s},Y_{s})\}, 
\]
is optimal for all $(t,x,y)\in[0,T]\times\cO$ and the process $(Z_s)_{s\in[t,T]}$ defined as
\[
Z_s:=\e^{-\int_t^{s}r(u,X_u,Y_u)\ud u }v(s,X_s,Y_s)
\]
is a super-martingale, whereas the stopped process $(Z_{s\wedge\tau_*})_{s\in[t,T]}$ is a martingale for all $(t,x,y)\in[0,T]\times\cO$.
Existence of an optimal stopping time and the (super-)martingale property of the value process are discussed extensively in, e.g., \cite{EK,S,PS} and \cite[Appendix D]{KS2}, while the finiteness of the value is generally easy to prove in specific examples (see, e.g., \cite{PS}) and it is immediate if, for example, $g$ is bounded.

As usual we denote the continuation set by
\[
\cC:=\big\{(t,x,y)\in[0,T]\times\cO: v(t,x,y)>g(t,x,y)\big\}
\]
and the stopping set by $\cS=\big([0,T]\times\cO\big)\setminus\cC$. We assume that $\cC\neq\varnothing$ and $\cS\neq\varnothing$.

\begin{remark}[Gain function]
It will be clear from the analysis below that adding a running cost/profit of the form
\[
\int_t^{\tau}\e^{-\int_t^s r(u,X_u,Y_u)\ud u}f(s,X_s,Y_s)\ud s
\]
in the optimisation criterion leads to no additional difficulty. Moreover, by an application of Dynkin's formula, the running term can usually be absorbed in the form \eqref{eq:osp-1} by replacing $g$ with $g+\tilde g$ for a function $\tilde g$ chosen so that $f=(\partial_t+\cL-r)\tilde g$.
\end{remark}

\subsection{Some comments on the operator $\cA$}
The family of processes described by the generator $\cL$ of the form \eqref{eq:generator} is quite general. The operator $\cL^\circ$ corresponds to the diffusive part of the dynamics of $(X,Y)$ whereas the operator $\cA$ is not required to have any specific form for the validity of the main results of this paper. However, our main motivating example is that of $\cA$ taking the form of a non-local integral operator, accounting for possible jumps in the dynamics of the process $(X,Y)$. In that case we should consider $\cA$ of the form
\begin{multline}\label{eq:cA}
(\cA f)(t,x,y):=\sum_{i=1}^L\int_{\R^d}\Big(f(t,(x,y) +\gamma_i(t, x, y, \xi))-f(t,x,y)\\
- \indd{\bar\gamma_i (\xi) < 1} \nabla f(t,x,y) \cdot \gamma_i(t, x, y, \xi)\Big)\nu_i(\ud \xi),
\end{multline}
where $L\in\mathbb N$ is fixed, functions $\gamma_i : [0, T] \times \cO \times \R^d \to \R^2$, $i = 1, \ldots, L$, are Borel measureable and $\nabla f \cdot \gamma_i$ stands for the $\R^2$-scalar product between $\nabla f := (\partial_x f, \partial_y f)$ and $\gamma_i$. Radon measures $\nu_i$, $i=1, \ldots, L$, are defined on $\R^d$ and satisfy
\begin{equation}\label{eqn:Levy_condition}
\int_{\R^d} (\bar\gamma_i^2(\xi) \wedge 1) \nu_i (\ud \xi) < \infty, 
\end{equation}
where $\bar \gamma_i$ is a Borel measurable function such that
\[
\sup_{(x,y) \in \cO} \sup_{t \ge 0} \|\gamma_i(t, x, y, \xi)\| \le \bar \gamma_i(\xi), \qquad \xi \in \R^d,
\]
with $\|\,\cdot\,\|$ denoting the Euclidean norm in $\R^2$.
The reader is referred to \cite[Chapter 2]{Garroni} for more general forms of integro-differential operators corresponding to jump processes with time-space dependent intensity and their properties.

Under specific assumptions (see \cite[Ch.~6]{Sato} and \cite[Sec.~1.3]{Oksendal}) generators of the form \eqref{eq:generator}, \eqref{eq:cL},\eqref{eq:cA} can be linked to solutions of L\'{e}vy SDEs. Here we briefly review some well-known facts about those SDEs. 

We write  $N_i$, $i = 1, \ldots, L$, for independent point processes on $\R^d$ with the L\'{e}vy measures $\nu_i$, defined on the filtered probability space $(\Omega,\cF, (\cF_t),\P)$. We define the compensated point process
\[
\tilde N_i(\ud t, \ud\xi) = N_i(\ud t,\ud\xi) - \indd{\bar\gamma_i(\xi) < 1} \nu_i(\ud\xi) \ud t.
\]
We assume that the probability space also supports two correlated Brownian motions $W,B$ with the correlation coefficient $\rho$ and independent of the point processes. Then, the generator $\cL$ with $\cA$ as in \eqref{eq:cA} can be linked to a solution of the following system of L\'{e}vy SDEs: for $s\ge t$
\begin{align*}
X_{s}&= x +\!\int_t^s\! \alpha_1(u,X_{u},Y_{u})\ud u\\
&\quad+\!\int_t^s\!\sqrt{2\beta_1(u,X_{u},Y_{u})}\ud B_{u}\!+ \!\sum_{i=1}^L \int_t^s\!\!\int_{\R^d}\! \gamma_i^{(1)}(u, X_{u-}, Y_{u-}, \xi) \tilde N_i(\ud u, \ud \xi),\notag\\
Y_{s}&= y+\!\int_t^s\! \alpha_2(u,X_{u},Y_{u})\ud u\\
&\quad+\!\int_t^s\! \sqrt{2\beta_2(u,X_{u},Y_{u})}\ud W_{u}\!+\! \sum_{i=1}^L \int_t^s\!\! \int_{\R^d}\! \gamma_i^{(2)}(u, X_{u-}, Y_{u-}, \xi) \tilde N_i(\ud u, \ud \xi),\notag 
\end{align*}
where $(\gamma_i^{(1)}, \gamma_i^{(2)})$ are the coordinates of $\gamma_i\in\R^2$. Detailed conditions for the existence and uniqueness of weak/strong solutions to the equations above can be found, for example, in \cite{Applebaum, Oksendal}. We will not dwell on these important results from stochastic analysis as they fall outside the scope of our work. Below we highlight two popular examples that are encompassed by our framework.  
\vspace{+5pt}

\noindent{(i) \em Compound Poisson process}. 
Our specification of the pure jump component encompasses a compound Poisson processes with time-dependent jump sizes. Then, for $s\ge t$, the process $(X, Y)$ solves
\begin{align*}
& \ud X_{s}=\alpha_1(s,X_{s},Y_{s})\ud s+\sqrt{2\beta_1(s,X_{s},Y_{s})}\ud B_{s}+ \kappa_1(s) \ud J_s, \quad X_t=x,\\
& \ud Y_{s}=\alpha_2(s,X_{s},Y_{s})\ud s+\sqrt{2\beta_2(s,X_{s},Y_{s})}\ud W_{s}+ \kappa_2(s) \ud K_s, \quad Y_t=y,
\end{align*}
for Borel measurable functions $\kappa_1, \kappa_2: [0, T] \to \R$ satisfying $\int_0^T |\kappa_i (s)| ds < \infty$, $i=1,2$, and a pair of independent compound Poisson process $(J,K)$. We note that the space $C^{0,2}_c((0,T) \times \cO)$ of functions with compact support and twice continuously differentiable with respect to the space variables is included in $D_\cA$.
\vspace{+5pt}

\noindent{(ii) \em Finite L\'evy measure}.
When the L\'{e}vy measure $\nu_i$ is finite, the corresponding point process $N_i$ has finite activity and $i$-th term in the generator $\cA$ from \eqref{eq:cA} can be written without the compensation as (see Example 3.3.7 in \cite{Applebaum})
\[
\int_{\R}\Big(f\big(t,(x,y) +\gamma_i(t, x, y, \xi)\big)-f(t,x,y)\Big)\nu_i(\ud \xi)
\]
with an appropriate adjustment to the drifts $\alpha_1, \alpha_2$. The domain $D_\cA$ contains the space of Borel measurable bounded functions (c.f. \cite[Sec.~2.1]{Garroni}).

\section{Continuity of the stopping boundary: the continuous case}\label{sec:con-1}

In this section we assume $\cA\equiv 0$ so that our process $(X,Y)$ has continuous paths and the associated infinitesimal generator \eqref{eq:generator} reduces to $\cL=\cL^\circ$. The operator $\cL^\circ$ has a local nature that allows us to prove our results under mild assumptions which are also local. For that reason we will often use the notation $\cU$ to indicate a generic bounded open hyperrectangle in $[0,T)\times\cO$ of the form
\begin{align}\label{eq:cU}
\cU=(t,s)\times (x_d, x_u) \times (y_d, y_u).
\end{align}

The main result of the section is the continuity of the optimal stopping boundary, i.e., $\partial\cC$, in any subset $\cU$ of the form above in which certain regularity conditions are verified. Letting $\cU$ be any such subset, we make four standing assumptions. The first one says that $\partial\cC$ can be locally represented as a surface with certain monotonicity properties, the second and third ones clarify the regularity required for the coefficients of the SDE and the gain function, the fourth one concerns regularity of the value function. 

\begin{assumption}
\label{ass:all}
Let $\cU$ be such that $v\in C(\cU)$, $\cU\cap \cC\neq \varnothing$, $\cU\cap\mathrm{int}(\cS)\neq\varnothing$ and the following conditions hold:
\begin{itemize}
\item[(i)]{\em(The boundary)}
There exists a function $(t,y)\mapsto x_{*}(t,y)$, such that
\begin{equation}
\label{eq:b-d}
\cC\cap\cU =\{(t,x,y)\in\cU: x > x_{*}(t,y)\}
\end{equation}
and both $t\mapsto x_{*}(t,y)$ and $y\mapsto x_{*}(t,y)$ are monotonic on their respective domains in $\cU$.\\

\item[(ii)] {\em(Coefficients of the SDE) }
For the coefficients of the SDE and the discount rate it holds 
\[
\alpha_i,\,\beta_i,\,r,\, \partial_x \beta_i \in C(\cU),
\]
for $i=1,2$. Moreover, $\beta_2>0$ on $\cU$.\\

\item[(iii)] {\em(The gain function)} We have $g\in C^{1,2}(\cU)$ and setting $h:=(\partial_t+\cL-r)g$ we have $\partial_x h\in C(\cU)$ and $\tfrac{\partial}{\partial x} (h/\beta_2)\ge \delta$ on $\cU$ for some $\delta>0$.\\

\item[(iv)] {\em(The value function)} We have $v\in C^1(\cU)$ with $\partial_x (v-g) \ge 0$ on $\cU$. Moreover, $v$ satisfies the boundary value problem
\begin{align}\label{eq:PDE}
(\partial_t+\cL-r)f=0\:\:\text{on $\cC\cap\mathcal U$,\quad with $f=v$ on $\partial (\cC\cap\mathcal U)$}
\end{align} 
with all derivatives understood in the {\em classical sense}.
\end{itemize}
\end{assumption}

\begin{definition}\label{def:cont}
We say that $x_{*}(\,\cdot\,,\,\cdot\,)$ is continuous on $\cU$ if the mapping $(t,y)\mapsto x_{*}(t,y)$ is continuous on each rectangle $D = (t_1, t_2) \times (y_1, y_2)$ such that $(t, x_*(t, y), y) \in \cU$ for all $(t,y)\in D$. 
\end{definition}

Notice that if $\beta_1,\beta_2>0$ on $\overline{\cU}$ and $\rho^2<1$ the operator $\cL=\cL^\circ$ is uniformly elliptic on $\overline{\cU}$, so that we can rely upon the next well-known lemma that guarantees \eqref{eq:PDE}.
\begin{lemma}\label{lem:PDE}
Let $\cU\subset[0,T)\times\cO$ be defined as in \eqref{eq:cU}. Assume that
\[
\alpha_1,\,\alpha_2,\,\beta_1,\,\beta_2,\,\bar \beta,\,r
\]
are H\"older continuous in $\cU$, $\beta_1,\beta_2>0$ in $\overline{\cU}$ and $\rho\in(-1,1)$. If $v$ is continuous then $v\in C^{1,2}(\cC\cap\mathcal U)$ and it solves the boundary value problem in \eqref{eq:PDE}. 
\end{lemma}
\begin{proof}
For any cylinder $D:= (t_0, t_1) \times B \subset \cC \cap \cU$, where $B$ is an open ball, consider a terminal boundary value problem
\begin{equation}\label{eqn:cyl}
(\partial_t+\cL^\circ-r)f =0\text{ on $D$, \quad with $f=v$ on $\partial D$,}
\end{equation}
where $\partial D := ([t_0, t_1] \times \partial B) \cup (\{t_1\} \times B)$ denotes the parabolic boundary. This problem admits a unique classical solution (see \cite[Ch.~3, Cor.~2, p.~71]{Fri}). 

Now, let $(D_n)_{n\in\N}$ be an increasing sequence of cylinders contained in $D$ and such that $D_n\uparrow D$ as $n\to\infty$. Let $\tau_{D}$ be the first exit time of $(t,X,Y)$ from $D$ and $\tau_{D_n}$ the first exit time from $D_n$. An application of It\^o's formula gives
\begin{align*}
f(t,x,y)&=\E_{t,x,y}\left[\e^{-\int_t^{\tau_{D_n}}r(s,X_s,Y_s)\ud s}f(\tau_{D_n},X_{\tau_{D_n}},Y_{\tau_{D_n}})\right].
\end{align*}
Let $n\to\infty$. Using the uniform ellipticity of $\cL^\circ$ on $\overline{D}$, we obtain that $\tau_{D_n}\uparrow \tau_D$ almost surely as $n\to\infty$. Since $D$ is bounded and $f$ is continuous, by the dominated convergence theorem we obtain the first equality below
\begin{align*}
f(t,x,y)&=\E_{t,x,y}\left[\e^{-\int_t^{\tau_{D}}r(s,X_s,Y_s)\ud s}f(\tau_{D},X_{\tau_{D}},Y_{\tau_{D}})\right]\\
&=\E_{t,x,y}\left[\e^{-\int_t^{\tau_D}r(s,X_s,Y_s)\ud s}v(\tau_D,X_{\tau_D},Y_{\tau_D})\right]=v(t,x,y),
\end{align*}
for all $(t,x,y)\in D$; the second equality follows from $(\tau_{D},X_{\tau_{D}},Y_{\tau_{D}}) \in \partial D$ and the final equality is by the martingale property of the value function. Hence, $v$ is a unique classical solution of \eqref{eqn:cyl}. As $D$ is arbitrary in $\cC \cap \cU$, we conclude that $v$ solves \eqref{eq:PDE}.
\end{proof}

Under Assumption \ref{ass:all} the function $u=v-g$ solves 
\begin{align}
\label{eq:pde-0}
&\big(\partial_t +\cL -r \big)u=-h, \quad\text{on $\cC\cap\cU$,}
\end{align}
with boundary conditions
\begin{align}\label{eq:smoothfit}
&u=\partial_t u=\partial_x u=\partial_y u=0, \quad\text{on $\partial\cC\cap\cU$}. 
\end{align}

Some comments on the assumptions above are in order. Continuity of the value function (at least locally) is generally not difficult to prove and there are numerous papers addressing this question in a broad generality (see, e.g., \cite{PS} and references therein for a rich class of examples; see also \cite{L09} or \cite{PS11} for problems with irregular data and \cite{JLL} for connections to variational problems). If $v\in C(\cU)$, then the well-posedness (in the sense above) of the optimal stopping problem usually leads to higher smoothness of the value function (e.g., as in Lemma \ref{lem:PDE}). 

The existence of an optimal boundary is normally proved on a case by case basis and it is known that there are several possible sufficient conditions that guarantee it (see, e.g., \cite{JL92} for an early contribution in this direction and, again, \cite{PS} for various techniques developed in specific examples). Therefore, rather than providing an inevitably incomplete list of such sufficient conditions we directly assume that the boundary exists. Also assuming that the continuation set lies above the boundary is with no loss of generality and the results of this paper carry over to the case in which $\cC$ lies below the boundary, up to obvious changes to the arguments of proof. Requiring \emph{local} monotonicity of the boundary is necessary to avoid pathological examples of boundaries with infinite local variation. In practice, the monotonicity is also checked on a case by case basis and sufficient conditions are known\footnote{For example, if $T<\infty$ and if $g$, $r$ and the coefficients of the SDE are independent of time, one immediately obtains that $t\mapsto (v-g)(t,x,y)$ is non-increasing. So, if \eqref{eq:b-d} holds the boundary is increasing in time.}. 

Local regularity of the coefficients of the SDE, the discount rate and the gain function are non-restrictive and hold in virtually all examples addressed in the optimal stopping literature. 
The condition $\partial_x(h/\beta_2)\ge \delta$ is of a slightly technical nature but it is in line with the fact that $\cC$ lies \jp{(locally)} above the optimal boundary. Indeed, notice that if $\partial_x\beta_2=0$, the condition is equivalent to $\partial_x h>0$. In many cases the latter is sufficient to prove\footnote{For example, take $g\in C^{1,2}([0,T]\times\cO)$ and $r(t,x,y)\equiv r>0$ then by an application of Dynkin's formula
\[
(v-g)(t,x,y)=\sup_{t\le \tau\le T}\E_{t,x,y}\Big[\int_t^\tau\e^{-rs}h(s,X_s,Y_s)\ud s \Big].
\]
Assume $\partial_x\alpha_2=\partial_x \beta_2=0$ and that $(X,Y)$ is a strong solution. Denote by $(X^{t,x,y},Y^{t,x,y})$ the process with the initial condition $(X_t,Y_t)=(x,y)$. For $x'>x$ we have, almost surely, $X^{t,x,y}_s\le X^{t,x',y}_s$ and $Y^{t,x,y}_s=Y^{t,x',y}_s$ for all $s\ge t$ by pathwise comparison. If $\partial_x h>0$ then $(v-g)(t,x',y) \ge (v-g)(t,x,y)$, which implies that the mapping $x\mapsto (v-g)(t,x,y)$ is increasing.} $\partial_x v\ge \partial_x g$ which then implies the existence of the boundary as in \eqref{eq:b-d}. 

Sufficient conditions that guarantee $v\in C^1(\cU)$ are provided in \cite{DeAPe20} and numerous extensions have been developed in specific examples (see, e.g., \cite{CDeAP21} and \cite{JP} for multi-dimensional optimal stopping problems). It is not hard to check that the requirement $\partial_x(v-g)\ge 0$ is equivalent to \eqref{eq:b-d}, since $\cU$ can be chosen arbitrarily small around a point of the boundary $\partial\cC$. Despite this slight redundancy we prefer to add the condition as part of our assumptions for clarity of exposition. 

Finally, \eqref{eq:PDE} holds under very mild conditions that are satisfied in all examples we are aware of. There are many sufficient conditions on the coefficients of the SDE that guarantee \eqref{eq:PDE} (as, for example, in Lemma \ref{lem:PDE} above) but we decided to state the assumptions in a broader generality to also cover some degenerate cases including, e.g., $\beta_1\equiv 0$ (as in the American Asian option \cite[Sec.27]{PS}) or even $\alpha_1=\beta_1\equiv 0$, where $x$ only enters as a parameter in connection to singular stochastic control problems (see, e.g., \cite{DeAFF17}). 

Let us now state the main result of this section.
\begin{theorem}
\label{them:conti-1}
Let $\cA\equiv 0$. Under Assumption \ref{ass:all}, the optimal stopping boundary $(t,y)\mapsto x_{*}(t,y)$ is continuous on $\cU$ in the sense of Definition \ref{def:cont}.
\end{theorem}

\begin{proof}
Since the maps $t\mapsto x_{*}(t,y)$ and $y\mapsto x_{*}(t,y)$ are monotonic it is sufficient to show that they are also continuous. Then the map $(t,y)\mapsto x_{*}(t,y)$ is continuous (see, e.g., \cite{KD69}).
In the rest of the proof we focus on showing continuity of $t\mapsto x_{*}(t,y)$ and $y\mapsto x_{*}(t,y)$. We fix a rectangle $D$ as in Definition \ref{def:cont} and restrict our attention to this domain.

We start by proving our claim in the case when $x_*$ is non-decreasing in both $t$ and $y$. By the continuity of $v$ on $\cU$ we know that $\cC\cap \cU$ is an open set and $\cS\cap \cU$ is closed relatively to $\cU$. Then we can conclude that $x_{*}(t,y)$ is right-continuous in $y$ for each $t$ and right-continuous in $t$ for each $y$. Indeed, fix $(t,y) \in D$ and let $y_n\downarrow y$ as $n\to\infty$ and such that $(t, y_n) \in D$. Then 
\begin{align}\label{eq:rc}
\cS \ni (t,x_{*}(t,y_n),y_n) \xrightarrow{n \to \infty} (t,x_{*}(t,y+), y) \in \cS,
\end{align}
where the limit exists by monotonicity and it lies in $\cS$ by the closedness of $\cS$. Using $(t,x_{*}(t,y+), y) \in \cS$ and \eqref{eq:b-d}, we conclude that $x_*(t,y+)\le x_{*}(t,y)$. Since $x_{*}(t,y_n)\ge x_{*}(t,y)$ for all $n$'s we conclude that $x_{*}(t,y+)= x_{*}(t,y)$. An analogous argument holds for the right-continuity in time.

Next we prove the left-continuity of $y\mapsto x_{*}(t,y)$. Let us fix $(t_0,y_0) \in D$ and, arguing by contradiction, let us assume $x_{*}(t_0,y_0-)<x_{*}(t_0,y_0)$. Take $x_{*}(t_0,y_0-)<x_1<x_2<x_{*}(t_0,y_0)$. By the assumed monotonicity of the boundary we have
\begin{align*}
\Sigma_{t_0}:=\{t_0\}\times(x_1,x_2)\times(\tilde{y},y_0)\subset\cC\cap\cU
\end{align*}
\jp{for some $\tilde{y}<y_0$ sufficiently close to $y_0$,} and
\[
\Sigma_{t_0,y_0}:=\{t_0\}\times (x_1,x_2)\times\{y_0\}\subset \cS\cap\cU.
\]

By (iii) and (iv) in Assumption \ref{ass:all} we have that $u:=v-g$ satisfies the PDE \eqref{eq:pde-0}.
In particular, we also have 
\begin{equation}
\label{eq:pde-2}
(\partial_t +\cL -r)u(t_0,x,y)=-h(t_0,x,y), \qquad \text{for $(t_0,x,y)\in\Sigma_{t_0}$}.
\end{equation}
Using that $\beta_2 > 0$ in $\cU$, we rearrange terms above and obtain
\begin{equation}\label{eqn:u_yy}
\partial_{yy} u=-\frac{1}{\beta_2}\big[h+\partial_t u+\beta_1\partial_{xx}u+2\bar\beta \partial_{xy}u+\alpha_1\partial_x u+\alpha_2 \partial_y u -ru\big],\quad\text{on $\Sigma_{t_0}$.}
\end{equation}

Take an arbitrary $\varphi\in C^\infty_c(x_1,x_2)$ with $\varphi\ge 0$ and $\int_{x_1}^{x_2}\varphi(x)\ud x = 1$. We multiply both sides of \eqref{eqn:u_yy} by $\partial_x \varphi$ and integrate over $(x_1, x_2)$:
\begin{align}\label{eq:uyyb}
&\int_{x_1}^{x_2} \partial_{yy} u (t_0, x, y) \partial_x \varphi(x) \ud x \\
&=
- \int_{x_1}^{x_2} \frac{1}{\beta_2(t_0, x, y)} \big[h+\partial_t u+\alpha_1\partial_x u+\alpha_2 \partial_y u -ru\big](t_0, x, y)\ \partial_x \varphi(x)\ud x \notag\\
&\hspace{11pt} - \int_{x_1}^{x_2} \frac{1}{\beta_2(t_0, x, y)} \big[\beta_1\partial_{xx}u+2\bar\beta \partial_{xy}u\big](t_0, x, y)\ \partial_x \varphi(x) \ud x\notag\\[3pt]
&=: J_1(t_0, y) + J_2 (t_0, y).\notag
\end{align}
By $C^1$ regularity of $u$ and \eqref{eq:smoothfit}, using dominated convergence we have
\begin{align}\label{eqn:J_1}
\lim_{y \uparrow y_0} J_1(t_0, y) 
= & - \int_{x_1}^{x_2} \frac{h}{\beta_2} (t_0, x, y_0) \ \partial_x \varphi(x) \ud x=  \int_{x_1}^{x_2} \partial_x \bigg(\frac{h}{\beta_2}\bigg) (t_0, x, y_0) \ \varphi(x) \ud x \ge \delta,
\end{align}
where in the second equality we integrated by parts and the inequality is by Assumption \ref{ass:all}(iii), since $\varphi\ge 0$ and integrates to one. We integrate the second term, $J_2$, by parts and again use \eqref{eq:smoothfit}
\begin{equation}\label{eqn:J_2}
\begin{aligned}
\lim_{y \uparrow y_0} J_2(t_0, y) 
&= 
\lim_{y \uparrow y_0} \int_{x_1}^{x_2} \partial_x \bigg( \frac{\beta_1}{\beta_2}(t_0, \cdot, y) \partial_x \varphi \bigg) (x)\  \partial_x u(t_0, x, y) \ud x\\
&\hspace{12pt} +
\lim_{y \uparrow y_0} \int_{x_1}^{x_2} \partial_x \bigg( \frac{\bar \beta}{\beta_2}(t_0, \cdot, y) \partial_x \varphi \bigg) (x)\  \partial_y u(t_0, x, y) 
\ud x
= 0.
\end{aligned}
\end{equation}
By \eqref{eq:uyyb} and the limits in \eqref{eqn:J_1}--\eqref{eqn:J_2} there is $\hat y$ such that for all $y \in [\hat y, y_0)$ we have
\begin{equation}\label{eqn:delta_half}
\int_{x_1}^{x_2} \partial_{yy} u (t_0, x, y) \partial_x \varphi(x) \ud x > \frac{\delta}{2}.
\end{equation}

For any sufficiently small $\eps>0$, integrating over $y$ twice and using Fubini's theorem we have
\begin{align*}
\tfrac{\delta}{4}(y_0-\eps-\hat y)^2
&<
\int_{\hat{y}}^{y_0-\eps}\int_{y}^{y_0-\eps}\int_{x_1}^{x_2} \partial_{yy}u(t_0,x,\zeta)\ \partial_x \varphi(x)\ud x\,\ud \zeta\, \ud y\\
&=
\int_{x_1}^{x_2} \partial_x \varphi(x) \int_{\hat{y}}^{y_0-\eps} \big[ \partial_y u(t_0, x, y_0 - \eps) - \partial_y u (t_0, x, y) \big] \ud y \ud x\\
&=
\int_{x_1}^{x_2} \partial_x \varphi(x) \big[ (y_0 - \eps - \hat y) \partial_y u(t_0, x, y_0 - \eps) - u(t_0, x, y_0 - \eps) + u(t_0, x, \hat y) \big] \ud x,
\end{align*}
where the inequality is by the estimate \eqref{eqn:delta_half}.
Letting $\eps\to 0$ and using the dominated convergence theorem, the fact that $u=\partial_y u=0$ on $\Sigma_{t_0,y_0}$ and integration by parts we obtain 
\begin{align}\label{eq:contr0}
\tfrac{\delta}{4}(y_0-\hat y)^2\le \int_{x_1}^{x_2}\partial_x \varphi(x)\, u(t_0,x,\hat{y}) \ud x
= - \int_{x_1}^{x_2}\varphi(x)\, \partial_x  u(t_0,x,\hat{y}) \ud x \le 0,
\end{align}
where the final inequality is by Assumption \ref{ass:all}(iv). This contradicts the assumption that $\delta > 0$, so completes the proof that $y\mapsto x_{*}(t,y)$ is left-continuous (and therefore continuous).

We now prove the continuity of the map $t\mapsto x_{*}(t,y)$. For $(t_0,y_0)$ as above we argue by contradiction assuming that $x_{*}(t_0-,y_0)< x_{*}(t_0,y_0)$. Take $x_{*}(t_0-,y_0)<x_1<x_2<x_{*}(t_0,y_0)$. 
There are $\tilde{y}<y_0$ and $\tilde t<t_0$ such that $(\tilde{t},t_0)\times(\tilde{y},y_0) \subset D$ and
\[
x_{*}(t_0,y_0)\ge x_{*}(t_0,\tilde{y})>x_2>x_1>x_{*}(t_0-,y_0)\ge x_{*}(s,\zeta),
\]
for all $(s,\zeta)\in (\tilde{t},t_0)\times(\tilde{y},y_0)$, where the existence of $\tilde y$ follows from the continuity of $y\mapsto x_{*}(t_0,y)$, whereas the final inequality holds thanks to the monotonicity of $x_{*}$.
Thus, recalling the notation used above, we have $\Sigma_{t_0} \subset \cS\cap\cU$ and 
\begin{align}\label{eq:Sigma}
\Sigma:=(\tilde{t},t_0)\times(x_1,x_2)\times(\tilde{y},y_0) \subset \cC\cap\cU.
\end{align}
We remark that the inclusion $\Sigma_{t_0} \subset \cS\cap\cU$ is crucial in the rest of our proof and it is obtained thanks to the continuity of $y\mapsto x_*(t,y)$.

Let us take two arbitrary functions $\varphi \in C^{\infty}_{c}(x_1,x_2)$ and $\psi \in C^{\infty}_{c}(\tilde{y},y_0)$, such that $\varphi,\psi\ge 0$ with $\int_{x_1}^{x_2}\varphi(x)dx=1$ and $\int_{\tilde{y}}^{y_0}\psi(y)dy=1$. Thanks to the inclusion \eqref{eq:Sigma}, the equation \eqref{eq:pde-0} is satisfied on $\Sigma$. We multiply both sides of it by $\varphi(x)\psi(y)$ and integrate over $(x_1,x_2)\times (\tilde{y},y_0)$ to obtain, for all $t\in(\tilde t, t_0)$
\begin{align*}
&\int_{x_1}^{x_2}\!\int_{\tilde{y}}^{y_0}\! \left(\partial_t u+\cL u-r u\right)(t,x,y)\, \varphi(x)\psi(y)\ud y\, \ud x
=-\int_{x_1}^{x_2}\!\int_{\tilde{y}}^{y_0}\!h(t,x,y)\, \varphi(x)\psi(y)\ud y\,\ud x.
\end{align*}
Using integration by parts for the terms of $\cL$ involving second derivatives we get
\begin{align}
\label{eq:pde-t-2}
&\int_{x_1}^{x_2}\!\int_{\tilde{y}}^{y_0}\!\big(\partial_t u +\alpha_1 \partial_x u+ \alpha_2 \partial_y u-ru\big) (t,x,y)\, \varphi(x)\psi(y) \ud y\, \ud x\notag\\
&-\int_{x_1}^{x_2}\!\int_{\tilde{y}}^{y_0}\!\big[\big(\partial_x(\varphi\beta_1)\partial_x u +2\partial_x(\varphi\bar \beta)\partial_y u\big) \psi+\partial_x(\psi\beta_2)\partial_y u\, \varphi \big](t,x,y)\ud y\,\ud x\\
&=-\int_{x_1}^{x_2}\int_{\tilde{y}}^{y_0}h(t,x,y)\, \varphi(x)\psi(y)\ud y\,\ud x,\notag
\end{align}
for all $t\in (\tilde{t},t_0)$.
Letting $t \uparrow t_0$, by the dominated convergence and the $C^1$-regularity of $u$ we obtain
\begin{equation*}
0=\int_{x_1}^{x_2}\int_{\tilde{y}}^{y_0}h(t_0,x,y)\,\varphi(x)\psi(y)\ud y\, \ud x. 
\end{equation*}
Since $\varphi,\psi$ are arbitrary, the latter implies that $h(t_0,x,y)=0$ for all $(x,y)\in (x_1,x_2)\times(\tilde{y}, y_0)$. This is also equivalent to $(h/\beta_2)(t_0,x,y)=0$ for all $(x,y)\in (x_1,x_2)\times(\tilde{y}, y_0)$ since $\beta_2>0$ on $\cU$. This contradicts Assumption \ref{ass:all}(iii). Therefore $t\mapsto x_{*}(t,y)$ is continuous.

We show how to adapt the proof for a different monotonicity of $t\mapsto x_{*}(t,y)$ and $y\mapsto x_{*}(t,y)$. Assume that $t\mapsto x_{*}(t,y)$ is non-decreasing but $y\mapsto x_{*}(t,y)$ is non-increasing. By analogous arguments to those in \eqref{eq:rc} we conclude that $y\mapsto x_{*}(t,y)$ is left-continuous and $t\mapsto x_{*}(t,y)$ is right-continuous.  Arguing by contradiction, we assume that $x_{*}(t_0,y_0)>x_{*}(t_0,y_0+)$ and, with an abuse of notation, we define
\[
\Sigma_{t_0}:= \{t_0\}\times (x_1,x_2)\times (y_0, \tilde y), \quad \text{and} \quad \Sigma_{t_0,y_0}:=\{t_0\}\times (x_1,x_2)\times\{y_0\}\subset \cS\cap\cU
\]
for some $\tilde{y}>y_0$ and $x_1, x_2$ satisfying $x_{*}(t_0,y_0)>x_2 > x_1 > x_{*}(t_0,y_0+)$. Using \eqref{eq:pde-2} with $\Sigma_{t_0}$ given above and repeating verbatim the arguments employed previously in the proof with $y \uparrow y_0$ replaced by $y\downarrow y_0$ in \eqref{eqn:J_1}--\eqref{eqn:J_2} we show that there exists $\hat{y}\in (y_0,\tilde{y})$ such that \eqref{eqn:delta_half} holds for any $y \in (y_0, \hat y]$.
Integrating \eqref{eqn:delta_half} over $y$ twice we obtain for $\eps > 0$ sufficiently small
\begin{align*}
\tfrac{\delta}{4}(\hat y -y_0-\eps)^2
&<\int^{\hat{y}}_{y_0+\eps}\int^{y}_{y_0+\eps}\int_{x_1}^{x_2} \partial_{yy}u(t_0,x,\zeta)\, \partial_x \varphi(x)\ud x\,\ud \zeta\, \ud y\\
&= \int_{x_1}^{x_2} \big[ u(t_0, x, \hat y) - u(t_0, x, y_0+\eps) - (y - y_0 - \eps) \partial_y u(t_0, x, y_0 + \eps) \big] \, \partial_x \varphi(x)\ud x.
\end{align*}
Letting $\eps\to 0$, using that $u=\partial_y u=0$ on $\partial\cC$, and integrating by parts we arrive at a contradiction similarly as in \eqref{eq:contr0}. Continuity of $t\mapsto x_*(t,y)$ can now be proven analogously to the argument used in \eqref{eq:pde-t-2} but paying attention to the fact that $\tilde y>y_0$ in this case and, therefore, obvious changes are required.

The remaining two cases, in which $t\mapsto x_*(t,y)$ is non-increasing and $y\mapsto x_*(t,y)$ is either non-increasing or non-decreasing, can be treated analogously.
\end{proof}

\jp{For the continuous case illustrated in this section it is worth fleshing out precisely additional difficulties related to our 3-dimensional set-up compared to the 2-dimensional set-ups considered in \cite{DeA15} and \cite{P19}.}
\begin{remark}\label{rem:goran1}
\jp{Continuity of time dependent boundaries $t\mapsto x_*(t)$ is shown in \cite[Thm.\ 3.1]{DeA15} and \cite[Thm.\ 3]{P19} without the requirement $\partial_x (h/\beta_2)\ge \delta$ in our Assumption \ref{ass:all}(iii). In the proofs of those theorems a contradiction can be reached relying on a much weaker requirement that $h\le -\delta$ in a neighbourhood of each point of $\partial\cC$. An analogous approach in our set-up does not seem to work because our contradiction stems from the integration by parts in \eqref{eq:contr0}, which cannot be obtained if we replace $\partial_x \varphi$ with $\varphi$ in the first step of our proof (Eq.\ \eqref{eq:uyyb}).}
\end{remark}
\begin{remark}\label{rem:goran2}
\jp{It is shown in \cite[Thm.\ 17]{P19} that a discontinuity of the first kind (\cite[Def.~2]{P19}) of a boundary $y\mapsto x_*(y)$ implies that the (horizontal) smooth-fit must hold. That is why \cite[Cor.\ 20]{P19} states the continuity of the boundary $y\mapsto x_*(y)$ without assuming smooth-fit but still relying on it in the proof. In our 3-dimensional setting, we replace smooth-fit with $C^1$-regularity of the value function in Assumption \ref{ass:all}(iv) (motivated by the theory in \cite{DeAPe20}). We notice that the arguments from \cite[Thm.\ 17]{P19} cannot be used to prove that a discontinuity of the first kind of a the mapping $y\mapsto x_*(t,y)$ induces (any sort of) smooth-fit, because the process $(t,X,Y)$ is not bound to evolve on a plane (as it would be required for the argument from the proof of \cite[Thm.\ 17]{P19}).}
\end{remark}
\begin{remark}\label{rem:goran3}
\jp{The continuity of a boundary $y\mapsto x_*(y)$ is proven in \cite[Thm.\ 10 and Cor.\ 20]{P19} without requiring its monotonicity. In our case the monotonicity (see Assumption \ref{ass:all}(i)) is crucial in order to lift the continuity of $t\mapsto x_*(t,y)$ and $y\mapsto x_*(t,y)$ to the joint continuity of $(t,y)\mapsto x_*(t,y)$ using \cite{KD69}.}
\end{remark}
\jp{It is also worth pointing the interested reader to \cite[Thm.\ 7 and 12, Cor.\ 18 and 21]{P19}, where, in the 2-dimensional set-up, various other sufficient conditions for the continuity of optimal boundaries $y\mapsto x_*(y)$ and $t\mapsto x_*(t)$ are presented.}


\section{Continuity of the stopping boundary: the general case}\label{sec:con-jump}

In this section we prove an analogue of Theorem \ref{them:conti-1} but in the case when the operator $\cA$ is non-null. Assumption \ref{ass:all} remains in force throughout this section. However, due to the inclusion of the additional term $\cA$ in the infinitesimal operator, we need to make the following additional assumption:
\begin{assumption}
\label{ass:all2}
The function $u := v - g$ is such that $u\in D_\cA$ and $\cA u \in C(\cU)$. Moreover, 
\begin{align}\label{eq:assA}
\partial_x \big(\beta_2^{-1} \cA u\big) (t, x, y) \ge 0,\quad\text{for $(t,x,y)\in \cU$ such that $x < x^*(t,y)$}.
\end{align}
\end{assumption}

Sufficient conditions for the above assumption are discussed at the end this section. 

\begin{theorem}
\label{them:conti-2}
Under Assumptions \ref{ass:all} and \ref{ass:all2}, the optimal stopping boundary $(t,y)\mapsto x_{*}(t,y)$ is continuous on $\cU$ in the sense of Definition \ref{def:cont}.
\end{theorem}
\begin{proof}
The proof follows the same ideas as the one of Theorem \ref{them:conti-1}, hence we only highlight small technical differences. We will only provide full arguments for the case in which both $t\mapsto x_*(t,y)$ and $y\mapsto x_*(t,y)$ are non-decreasing. All the remaining cases can be treated analogously as shown at the end of the proof of Theorem \ref{them:conti-1}.

We fix a rectangle $D$ as in the Definition \ref{def:cont} and restrict our attention to this domain. From \eqref{eq:rc} we obtain that $t\mapsto x_*(t,y)$ and $y\mapsto x_*(t,y)$ are right-continuous. To prove left-continuity in $x$, we fix $(t_0,y_0) \in D$ and, arguing by contradiction, we assume $x_{*}(t_0,y_0-)<x_{*}(t_0,y_0)$. We take $x_{*}(t_0,y_0-)<x_1<x_2<x_{*}(t_0,y_0)$ and by the monotonicity of the boundary we have 
\begin{align*}
\Sigma:=(\tilde{t},t_0)\times(x_1,x_2)\times(\tilde{y},y_0) \subset \cC\cap\cU\,,\quad \Sigma_{t_0}:=\{t_0\}\times(x_1,x_2)\times(\tilde y,y_0)\subset\cC\cap\cU
\end{align*}
for some $\tilde{y}<y_0$ and $\tilde{t}<t_0$ sufficiently close to $(t_0,y_0)$, and
\[
\Sigma_{t_0,y_0}:=\{t_0\}\times (x_1,x_2)\times\{y_0\}\subset \cS\cap\cU.
\]

Equation \eqref{eqn:u_yy} has an additional term due to the operator $\cA$:
\begin{equation*}
\partial_{yy} u=-\frac{1}{\beta_2}\big[h+\partial_t u+\beta_1\partial_{xx}u+2\bar\beta \partial_{xy}u+\alpha_1\partial_x u+\alpha_2 \partial_y u -ru+ \cA u\big],\quad\text{on $\Sigma_{t_0}$}.
\end{equation*}
We add this term to $J_1$ in \eqref{eq:uyyb} and obtain an analogue of \eqref{eqn:J_1} by the dominated convergence:
\begin{equation*}
\begin{aligned}
\lim_{y \uparrow y_0} J_1(t_0, y) 
&= - \int_{x_1}^{x_2} \frac{h + \cA u}{\beta_2} (t_0, x, y_0) \ \partial_x \varphi(x) \ud x \\
&= \int_{x_1}^{x_2} \partial_x \bigg(\frac{h}{\beta_2}\bigg) (t_0, x, y_0) \ \varphi(x) \ud x 
+
\int_{x_1}^{x_2} \partial_x \bigg(\frac{\cA u}{\beta_2}\bigg) (t_0, x, y_0) \ \varphi(x) \ud x
\ge \delta,
\end{aligned}
\end{equation*}
where the first equality is by the continuity of $h, \cA u, \beta_2$ on $\cU$. The first term in the second line above is bounded from below by $\delta$ thanks to Assumption \ref{ass:all}(iii). The second term is non-negative by Assumption \ref{ass:all2} and the non-negativity of $\varphi$. The rest of the proof continues as in Theorem \ref{them:conti-1} and eventually we obtain the contradiction in \eqref{eq:contr0}. Hence $y\mapsto x_*(t_0,y)$ is continuous.

For the continuity of $t\mapsto x_*(t,y)$, using once again that $y\mapsto x_*(t,y)$ is continuous we can guarantee that $\Sigma_{t_0}\subset \cS\cap\cU$ for $\tilde y<y_0$ and sufficiently close. The next part of the proof differs slightly from the one in Theorem \ref{them:conti-1}. 

Let us take two arbitrary functions $\varphi \in C^{\infty}_{c}(x_1,x_2)$ and $\psi \in C^{\infty}_{c}(\tilde{y},y_0)$, such that $\varphi,\psi\ge 0$ with $\int_{x_1}^{x_2}\varphi(x)dx=1$ and $\int_{\tilde{y}}^{y_0}\psi(y)dy=1$. We multiply the equation
\[
(\partial_t+\cL-r)u=-h\quad\text{in $\cC\cap\mathcal U$}
\]
by $\tfrac{1}{\beta_2} \psi\, \partial_x \varphi $ and integrate over $(x_1,x_2)\times (\tilde{y},y_0)$ to obtain, for all $t\in(\tilde t, t_0)$
\begin{equation}\label{eqn:phi_psi}
\begin{aligned}
&\int_{x_1}^{x_2}\!\!\!\int_{\tilde y}^{y_0}\!\partial_x \varphi(x)\, \psi(y)\big[\frac{1}{\beta_2}(\partial_t\!+\!\cL^\circ-r)u\big](t,x,y)\ud x\ud y\\
&\hspace{12pt}+\!\int_{x_1}^{x_2}\!\!\!\int_{\tilde y}^{y_0}\! \partial_x \varphi(x)\, \psi(y) \frac{\cA u}{\beta_2}(t,x,y)\ud x\ud y\\
&=-\!\int_{x_1}^{x_2}\!\!\!\int_{\tilde y}^{y_0}\!\partial_x \varphi(x)\, \psi(y) \frac{h}{\beta_2}(t,x,y)\ud x\ud y.
\end{aligned}
\end{equation}
For the first integral on the left-hand side above, integration by parts of the second derivatives gives
\begin{align*}
I_1(t):=&\int_{x_1}^{x_2}\!\!\!\int_{\tilde y}^{y_0}\!\partial_x \varphi(x)\, \psi(y)\big[\frac{1}{\beta_2}(\partial_t\!+\!\cL^\circ-r)u\big](t,x,y)\ud x\ud y\\
=&\int_{x_1}^{x_2}\!\!\!\int_{\tilde y}^{y_0}\!\partial_x \varphi(x)\, \psi(y)\big[\frac{1}{\beta_2}(\partial_t +  \alpha_1 \partial_x + \alpha_2 \partial_y - r)u\big](t,x,y)\ud x\ud y\\
&- \int_{x_1}^{x_2}\!\!\!\int_{\tilde y}^{y_0}\! \psi(y)\big[\partial_{x} \big(\beta_1/\beta_2\,\partial_x \varphi\big) \partial_x u + 2 \partial_x \big(\bar \beta/ \beta_2\,\partial_x \varphi \big) \partial_y u \big](t,x,y)\ud x\ud y\\
&- \int_{x_1}^{x_2}\!\!\!\int_{\tilde y}^{y_0}\! \partial_x \varphi(x)\, \partial_y \psi(y)\, \partial_y u(t,x,y)\ud x\ud y.
\end{align*}
Thus, by the dominated convergence theorem, letting $t\uparrow t_0$ we have $I_1(t)\to 0$, which will be used later.
We denote the the second term in \eqref{eqn:phi_psi} by
\[
I_2(t):=\int_{x_1}^{x_2}\!\!\!\int_{\tilde y}^{y_0}\! \partial_x \varphi(x)\, \psi(y) \frac{\cA u}{\beta_2}(t,x,y)\ud x\ud y.
\]
Recalling that $\cA u$ and $\beta_2$ are continuous we first take the limit as $t \uparrow t_0$ and then integrate by parts to obtain
\begin{equation}\label{eqn:I_2_jump}
\begin{aligned}
\lim_{t \uparrow t_0} I_2(t)
&=
\int_{x_1}^{x_2}\!\!\!\int_{\tilde y}^{y_0}\! \partial_x \varphi(x)\, \psi(y) \frac{\cA u}{\beta_2}(t_0,x,y)\ud x\ud y\\
&=
- \int_{x_1}^{x_2}\!\!\!\int_{\tilde y}^{y_0}\! \varphi(x)\, \psi(y) \partial_x \Big(\frac{\cA u}{\beta_2}\Big)(t_0,x,y)\ud x\ud y
\le 0,
\end{aligned}
\end{equation}
where the last inequality is by Assumption \ref{ass:all2} since the integration is over $\Sigma_{t_0} \subset \cS \setminus \graph(x^*)$ ($\graph(x^*)$ denotes the graph of $x^*$) and $\varphi, \psi \ge 0$.

For the expression on the right-hand side of \eqref{eqn:phi_psi} we have 
\begin{align*}
I_3(t):=&-\!\int_{x_1}^{x_2}\!\!\!\int_{\tilde y}^{y_0}\!\partial_x \varphi(x)\, \psi(y) \frac{h}{\beta_2}(t,x,y)\ud x\ud y\\
=&\!\int_{x_1}^{x_2}\!\!\!\int_{\tilde y}^{y_0}\!\varphi(x)\, \psi(y) \partial_x \Big(\frac{h}{\beta_2}\Big)(t,x,y)\ud x\ud y \ge \delta,
\end{align*}
where the last inequality is by Assumption \ref{ass:all}(iii). Using the dominated convergence theorem we obtain $\liminf_{t\uparrow t_0}I_3(t) \ge \delta$.

Summarising, we have shown that
\[
\lim_{t\uparrow t_0}I_1(t)=0,\quad \lim_{t\uparrow t_0}I_2(t) \le 0,\quad \liminf_{t\uparrow t_0}I_3(t) \ge \delta.
\]
Combining the above with the fact that $I_1(t)+I_2(t)=I_3(t)$ for all $t\in(\tilde t, t_0)$ leads to a contradiction.
\end{proof}

\subsection{Sufficient conditions for jump diffusions}
Below we give sufficient conditions for Assumption \ref{ass:all2} when the operator $\cA$ is of the form \eqref{eq:cA}. For the clarity of exposition, we present results for the case of one jump component, i.e., $L=1$. An extension to $L > 1$ is immediate. We make the following assumption:

\begin{assumption}\label{ass:jumps}
The function $\gamma_1$ is continuous in $(t,x,y)$ for each fixed $\xi \in \R^d$, $u=v-g\in C^{0,1}([0,T)\times\cO)$ and one of the following holds:
\begin{itemize}
 \item[(A.i)] the function $u$ is bounded;
 \item[(A.ii)] $\nu_1$ has a compact support and $\gamma_1$ is bounded on compact subsets of $[0, T] \times \cO \times \R^d$,
 \item[(A.iii)] $(t,x,y) \mapsto \sup_{\xi \in \R^d} \| \gamma_1(t,x,y, \xi)\|$ is bounded for $(t, x, y)$ in a compact set.
\end{itemize}
\end{assumption}

The next lemma provides sufficient conditions implying $u\in D_\cA$.

\begin{lemma}\label{lem:sc1}
Let Assumptions \ref{ass:jumps} hold. If $\nu_1$ satisfies 
\begin{align}\label{eq:integr-g}
\int_{\R^d} (\bar\gamma_1(\xi) \wedge 1)\, \nu_1(\ud \xi) < \infty,
\end{align}
then $\cA u$ is well defined (i.e.,\ $u\in D_\cA$) and $\cA u \in C([0, T) \times \cO)$. 
\end{lemma}
\begin{proof}
We rewrite $\cA u$ as follows:
\begin{equation}
\label{eqn:cA_terms}
\begin{aligned}
&(\cA u) (t, x, y) \\
&=\! 
\int_{\{\xi \in \R^d:\, \bar \gamma_1(\xi)< 1\}}\Big(u(t,(x,y)\! +\!\gamma_1(t, x, y, \xi))\!-\!u(t,x,y)
\!-\! \nabla u(t,x,y) \cdot \gamma_1(t, x, y, \xi)\Big)\nu_1(\ud \xi)\\
&\hspace{11pt}+
\int_{\{\xi \in \R^d:\, \bar \gamma_1(\xi)\ge 1\}}\Big(u(t,(x,y) +\gamma_1(t, x, y, \xi))-u(t,x,y)
\Big)\nu_1(\ud \xi)\\
&= (I_1) + (I_2).
\end{aligned}
\end{equation}
Let us first consider $(I_2)$. Using \eqref{eq:integr-g}, Assumption \ref{ass:jumps}(A.i) implies that $(I_2)$ is well-defined and, by the dominated convergence theorem, depends continuously on $(t, x, y) \in [0,T) \times \cO$. For any compact subset $Z \subset [0, T) \times \cO$, either Assumption \ref{ass:jumps} (A.ii) or (A.iii) yields that the set
\[
\big\{ \big(t,(x,y) +\gamma_1(t, x, y, \xi)\big) :\ \xi \in \supp (\nu_1), (t, x, y) \in Z \big\},
\]
is bounded. Combining this with the continuity of $u$ (hence, boundedness on compacts) shows that the term $(I_2)$ in \eqref{eqn:cA_terms} is well-defined and the dominated convergence theorem can be applied to deduce its continuous dependence on $(t, x, y)$.

Next we study term $(I_1)$ of \eqref{eqn:cA_terms} using ideas from \cite[Ch.\ 2]{Garroni}. We apply Taylor's formula to obtain
\begin{align*}
&\int_{\{\xi \in \R^d:\, \bar \gamma_1(\xi) < 1\}}\Big|u(t,(x,y)\! +\!\gamma_1(t, x, y, \xi))\!-\!u(t,x,y)
\!-\! \nabla u(t,x,y) \cdot \gamma_1(t, x, y, \xi)\Big|\ \nu_1(\ud \xi)\\
&=\!
\int_{\{\xi \in \R^d:\, \bar \gamma_1(\xi) < 1\}} \bigg|\int_{0}^1 \Big(\nabla u(t,(x,y)\! +\! \theta \gamma_1(t, x, y, \xi))\! -\! \nabla u(t,x,y) \Big) \cdot \gamma_1(t, x, y, \xi) \ \ud \theta \bigg|\, \nu_1(\ud \xi)\\
&\le\!
\int_{\{\xi \in \R^d:\, \bar \gamma_1(\xi) < 1\}} \int_{0}^1 \big\|\nabla u(t,(x,y)\! +\! \theta \gamma_1(t, x, y, \xi))\! -\! \nabla u(t,x,y) \big\|  \|\gamma_1(t, x, y, \xi)\| \ \ud \theta\, \nu_1(\ud \xi)\\
&\le\!
\int_{\{\xi \in \R^d:\, \bar \gamma_1(\xi) < 1\}} \int_{0}^1 \big\|\nabla u(t,(x,y)\! +\! \theta \gamma_1(t, x, y, \xi))\! -\! \nabla u(t,x,y) \big\|\, \bar\gamma_1(\xi) \ \ud \theta\, \nu_1(\ud \xi),
\end{align*}
where we used Cauchy-Schwartz inequality in the first inequality and the bound $\|\gamma_1\| \le \bar \gamma_1$ in the last inequality (recall also that $\|\,\cdot\,\|$ is the Euclidean norm in $\R^2$).
Since the integration is over $\xi$ such that $\bar \gamma_1(\xi) < 1$ and $u \in C^{0,1}([0,T)\times\cO)$, we conclude that the norm of the difference of gradients is bounded. By assumption, $\int_{\{\xi \in \R^d:\, \bar \gamma_1(\xi) < 1\}} \bar\gamma_1(\xi) \ \nu_1(\ud \xi) < \infty$, so in combination with the above this completes the proof that $(I_1)(t,x,y)<\infty$ for all $(t,x,y)\in[0,T)\times\cO$. Since the integration is over $\xi$ such that $\big(t,(x,y)+\theta\gamma_1(t,x,y,\xi)\big)$ is in a compact set, a uniform bound is available for the norm of the difference between gradients. This implies the continuity of term $(I_1)$ of \eqref{eqn:cA_terms} in $(t, x, y)$ by the dominated convergence.
\end{proof}

While the results above give us sufficient conditions for the first part of Assumption \ref{ass:all2}, the next lemma gives us sufficient conditions for \eqref{eq:assA}. Thus combining Lemma \ref{lem:sc1} and the next one, we have sufficient conditions for Assumption \ref{ass:all2}.
\begin{lemma}\label{lem:cA_der}
Let Assumption \ref{ass:all} and Assumption \ref{ass:jumps} hold with the exclusion of {\em(A.i)}. Let $u=v-g\in D_\cA$ and assume further that 
\begin{itemize}
 \item[(a)] $\partial_x \beta_2\le 0$,
 \item[(b)] $x \mapsto \gamma_1^{(1)}(t,x,y,\xi)$ (the first coordinate of $\gamma_1$) is continuously differentiable with $\partial_x \gamma^{(1)}_1 \ge -1$,
 \item[(c)] $\gamma_1^{(2)}$ (the second coordinate of $\gamma_1$) does not depend on $x$, i.e. $\partial_x \gamma^{(2)}_1 =0$, 
 \item[(d)] $\partial_x u$ is non-negative on $[0,T)\times\cO$.
\end{itemize}
Then for $(t,x,y)\in \cU$ such that $x < x^*(t,y)$ we have
\begin{equation}\label{eqn:cA_der}
\partial_x (\cA u)(t,x,y) = \int_{\R^d} \partial_x u \Big(t,(x,y) +\gamma_1(t, x, y, \xi)\Big)\, \big(1 + \partial_x \gamma_1^{(1)}(t, x, y, \xi)\big) \, \nu_1(\ud \xi) \ge 0
\end{equation}
and 
\[
\partial_x \big(\beta_2^{-1} \cA u\big) (t, x, y) \ge 0. 
\]
\end{lemma}
\begin{proof}
Take $(t,x,y)\in \cU$ such that $x < x^*(t,y)$, i.e., $(t,x,y) \in \cS$. Since $u = \partial_x u = \partial_y u = 0$ on $\cS$, we have
\begin{equation}\label{eqn:cA_on_S}
(\cA u)(t, x, y) = \int_{\R^d} u\big(t,(x,y) +\gamma_1(t, x, y, \xi)\big) \nu_1(\ud \xi).
\end{equation}
For all $\eps > 0$ sufficiently small we have $(t, x \pm \eps, y) \in \cS \cap \cU$. This allows us to bound the left and right derivatives in $x$. Indeed, for the right derivative we have
\begin{align*}
&\frac{(\cA u)(t, x+\eps, y) - (\cA u)(t, x, y)}{\eps} \\
&=
\int_{\R^d} \frac{1}{\eps}\Big( u\big(t,(x+\eps,y) +\gamma_1(t, x+\eps, y, \xi)\big) - u\big(t,(x,y) +\gamma_1(t, x, y, \xi)\big) \Big)\nu_1(\ud \xi)\\
&=
\int_{\R^d} \frac{1}{\eps}\int_0^1  \partial_x u \Big(t,(x,y) +\gamma_1(t, x, y, \xi) + \theta (\zeta (\xi, \eps), 0)\Big)\, \zeta(\xi, \eps)\ \ud \theta\, \nu_1(\ud \xi),
\end{align*}
where 
\begin{equation}\label{eqn:zeta}
\zeta (\xi, \eps)= \eps + \gamma_1^{(1)}(t, x+\eps, y, \xi) -  \gamma_1^{(1)}(t, x, y, \xi) 
\end{equation}
and in the second equality we used that $\gamma_1^{(2)}$ does not depend on $x$. Due to the continuity of $\gamma_1$ and of $\partial_x u$, we obtain
\[
\lim_{\eps \downarrow 0} \partial_x u \big(t,(x,y) +\gamma_1(t, x, y, \xi) + \theta (\zeta (\xi, \eps), 0)\big)
=
\partial_x u \big(t,(x,y) +\gamma_1(t, x, y, \xi)\big).
\]
By (b),
\[
\lim_{\eps \downarrow 0} \frac{\zeta(\xi,\eps)}{\eps} = 1 + \partial_x \gamma_1^{(1)}(t, x, y, \xi).
\]
By the dominated convergence theorem, enabled by Assumption \ref{ass:jumps} (A.ii) or (A.iii), we have
\begin{multline}\label{eqn:cA_right_der}
\lim_{\eps \downarrow 0} \frac{(\cA u)(t, x+\eps, y) - (\cA u)(t, x, y)}{\eps}\\
=
\int_{\R^d} \partial_x u \Big(t,(x,y) +\gamma_1(t, x, y, \xi)\Big)\, \big(1 + \partial_x \gamma_1^{(1)}(t, x, y, \xi)\big) \, \nu_1(\ud \xi).
\end{multline}
For the left derivative we apply the same technique and obtain
\begin{align*}
&\frac{(\cA u)(t, x-\eps, y) - (\cA u)(t, x, y)}{-\eps} \\
&=
- \int_{\R^d} \frac{1}{\eps}\Big( u\big(t,(x-\eps,y) +\gamma_1(t, x-\eps, y, \xi)\big) - u\big(t,(x,y) +\gamma_1(t, x, y, \xi)\big) \Big)\nu_1(\ud \xi)\\
&=
- \int_{\R^d} \frac{1}{\eps}\int_0^1 \Big( \partial_x u \big(t,(x,y) +\gamma_1(t, x, y, \xi) + \theta (\zeta (\xi, -\eps), 0)\big) \Big) \zeta(\xi, -\eps)\ \ud \theta\, \nu_1(\ud \xi),
\end{align*}
where $\zeta(\cdot)$ is defined in \eqref{eqn:zeta}. Taking $\eps \downarrow 0$, by the dominated convergence theorem, we obtain the same limit as in \eqref{eqn:cA_right_der}. This proves that $(\cA u)(t,x,y)$ is differentiable in $x$ and the derivative is given by the expression \eqref{eqn:cA_der}. It remains to notice that by (d), the partial derivative $\partial_x u$ is non-negative and thanks to (b) we have $1 + \partial_x \gamma_1^{(1)}(t, x, y, \xi) \ge 0$. This implies that $\partial_x \cA(t, x, y) \ge 0$.

To complete the proof, we note that
\[
\partial_x (\beta_2^{-1} \cA u) (t,x,y) = \frac{\partial_x (\cA u) \beta_2 - (\cA u) \partial_x \beta_2}{\beta_2^2}(t,x,y)
\]
and use that $\beta_2 > 0$, $\partial_x \beta_2 \le 0$ (by (a)), $\partial_x \cA u \ge 0$ (from \eqref{eqn:cA_der}) and $(\cA u) \ge 0$ since $u \ge 0$ and $\cA u$ is given by \eqref{eqn:cA_on_S}.
\end{proof}

If $\gamma_1 = \gamma_1(t,y,\xi)$, i.e., jump sizes do not depend on $x$, the statement of Lemma \ref{lem:cA_der} simplifies as follows.
\begin{corollary}
Let Assumption \ref{ass:all} and Assumption \ref{ass:jumps} with exclusion of {\em(A.i)} hold. Let $u=v-g\in D_\cA$ and
assume that
\begin{itemize}
 \item[(a)] $\partial_x \beta_2\le 0$,
 \item[(b)] $\gamma_1$ does not depend on $x$,
 \item[(c)] $\partial_x u$ is non-negative on $[0,T)\times\cO$.
\end{itemize}
Then for $(t,x,y)\in \cU$ such that $x < x^*(t,y)$ we have
\begin{equation*}
\partial_x (\cA u)(t,x,y) = \int_{\R^d} \partial_x u \big(t,(x,y) +\gamma_1(t, y, \xi)\big) \, \nu_1(\ud \xi) \ge 0
\end{equation*}
and 
$\partial_x \big(\beta_2^{-1} \cA u\big) (t, x, y) \ge 0$.
\end{corollary}

\begin{remark}
It is clear from the proof of Lemma \ref{lem:cA_der} that condition (d) can be relaxed when more is known about the direction and size of jumps. Indeed, it is only necessary that $\partial_x u \ge 0$ at points $(t,x,y)$ to which jumps can occur from the stopping set. Hence, condition (d) can be replaced by requiring that $\partial_x u \ge 0$ on the set
\[
\bigcup_{(t,x,y) \in \cS \cap \cU} \big\{ \big(t,(x,y) +\gamma_1(t, x, y, \xi)\big) :\ \xi \in \supp (\nu_1)\big\}.
\] 
\end{remark}

\section{Concluding remarks}\label{sec:last}

It is clear that our framework includes also simpler cases in which the optimal boundary is a function of a single variable. In those cases $r$, $v$ and $g$ are functions of two variables, rather than three: $r(t,x)$, $v(t,x)$ and $g(t,x)$, or $r(x,y)$, $v(x,y)$ and $g(x,y)$. Hence they fit within our framework by fictitiously adding a third variable, with respect to which $r$, $v$, $g$ and the optimal boundary are constant. \jp{We can therefore obtain the} continuity of boundaries of the form $y\mapsto x_*(y)$ and $t\mapsto x_*(t)$ as in \cite{P19} and \cite{DeA15,LM}, respectively, \jp{if the fictitious problem satisfies our assumptions. It should be noted that our assumptions correspond to different conditions for the two-dimensional problem than those required in the aforementioned papers.}

A close inspection of our arguments of proof suggests that in order to obtain continuity in dimension higher than what we can cover with Theorems \ref{them:conti-1} and \ref{them:conti-2} one would need to require continuity of \jp{some} second order derivatives of the value function up to the boundary. \jp{Indeed, if we add one more state variable, say $Z$, with diffusive dynamics, then on the right-hand side of \eqref{eqn:u_yy} we need to deal with terms of the form $\partial_{zz} u$ and $\partial_{yz} u$ that cannot be removed using integration by parts in $x$ as we do in \eqref{eq:uyyb}. Continuity up to the boundary of second order derivatives} is not known to hold in general for time-inhomogeneous optimal stopping problems with multi-dimensional diffusions and therefore we prefer to set the question aside for further research.

In some applications of optimal stopping theory it is necessary to consider more general additive terms in the underlying dynamics. In particular, setting $\cA\equiv0$ for simplicity, we may need to consider SDEs of the form
\begin{align*}
& \ud X_t=\alpha_1(t,X_t,Y_t)\ud t+\sqrt{2\beta_1(t,X_t,Y_t)}\ud B_t+\ud A_t, \quad X_0=x,\\
& \ud Y_t=\alpha_2(t,X_t,Y_t)\ud t+\sqrt{2\beta_2(t,X_t,Y_t)}\ud W_t+\ud C_t, \quad Y_0=y,
\end{align*}
where the processes $(A_t)$ and $(C_t)$ are of bounded variation (continuous) and take the form of additive functionals of the triple $(t,X,Y)$ (e.g., local times and/or running maximum/minimum of the process $X$ or $Y$). Likewise, we may add a running cost/profit and a more general discount factor as in
\begin{align*}
v(t,x,y)=\sup_{t\le \tau\le T}\E_{t,x,y}\left[\int_t^{\tau}\e^{-\Lambda_s}h(s,X_s,Y_s)\ud(s+ G_s)+\e^{-\Lambda_\tau}g(\tau,X_\tau, Y_\tau)\right],
\end{align*}
where $\Lambda_s:=\int_t^s r(u,X_u,Y_u)\ud u +H_s$ and the processes $(G_t)$ and $(H_t)$ are again of bounded variation (continuous) and in the form of additive functionals of $(t,X,Y)$. This is the case for example in the Russian option problem (see, e.g., \cite{DKvS, E04, P05, SS93}), where $Y$ is a Brownian motion, $X$ its running maximum, the interest rate is constant, $h\equiv 0$, $g=1$ and $H_s=X_s$. More examples appear in recent developments of the optimal dividend problem (see, e.g., \cite{DeAE17}), where $(X,Y)$ is typically a reflecting diffusion (so that $(A,C)$ are in the form of local times), $h\equiv 0$, $g=1$, the discount rate is constant and $(H_s)$ is also in the form of a local time of the process $(X,Y)$. These situations can be addressed with our Theorem \ref{them:conti-1} provided that the additional bounded variation processes $(A,C,G,H)$ are not supported on the domain $\cU$ used to define Assumption \ref{ass:all}. That is we need $\ud A_t=\ud C_t=\ud G_t=\ud H_t=0$ a.s.\ on $\cU$ (notice that indeed the definition of the infinitesimal generator $\cL$ is only needed locally on $\cU$). \jp{Examples studied in} the papers mentioned above fall under this class of processes and our results \jp{can be applied provided that Assumption \ref{ass:all} also holds}.

\end{document}